\documentclass[12pt]{amsart}

\usepackage[psamsfonts]{amssymb}
\usepackage{amscd}
\usepackage{amsthm}
\usepackage{type1cm}

\newtheorem{thm}{Theorem}

\newtheorem{thmA}{Theorem}

\newtheorem{lem}{Lemma}[section]
\newtheorem{cor}{Corollary}
\newtheorem{prop}{Proposition}[section]
\theoremstyle{remark}
\newtheorem*{rem}{{\bf Remark}}
\theoremstyle{definition}

\allowdisplaybreaks[4]

\newcommand{\Lip}{\operatorname{Lip}}

\begin{document}

\title[Compact homomorphisms on Lipschitz algebras]
{\boldmath Compact homomorphisms between algebras of $C(K)$-valued Lipschitz functions}%論文タイトル

%{Compact homomorphisms between algebras of Lipschitz functions with value in continuous functions}

\author{Shinnosuke Izumi}%第一著者名
\author{Hiroyuki Takagi$^\dag$}%第二著者名 
\thanks{$\dag$Deceased 25 November 2017. }%第二著者に関する脚注

\address{Shinnosuke Izumi, Department of Mathematics and System Development, Interdisciplinary Graduate school of Science and Technology,Shinshu University, Matsumoto, 390-8621, Japan}%住所
\email{17st201h@shinshu-u.ac.jp}%メールアドレス

\address{Hiroyuki Takagi, Department of Mathematics, Faculty of Science, Shinshu University, Matsumoto, 390-8621, Japan}%住所
%\email{takagi@math.shinshu-u.ac.jp}%メールアドレス

%キーワード


\begin{abstract}%アブストラクト
We give a complete description of homomorphisms between two Banach algebras of Lipschitz functions with values in continuous functions.
We also characterize the compactness of those homomorphisms.
\end{abstract}

\maketitle

\section{Introduction}%第１章

Let $X$ be a compact metric space with metric $d_X$ and $\mathcal{A}$ a commutative Banach algebra with norm $\| \cdot \|_{\mathcal{A}}$.
By $C(X, \mathcal{A})$, 
we denote the Banach algebra of all $\mathcal{A}$-valued continuous functions on $X$, with norm
$$
 \| f \|_{C(X, \mathcal{A})} = \sup \bigl\{ \| f(x) \|_{\mathcal{A}} : x \in X \bigr\}.
$$
If an $\mathcal{A}$-valued function $f$ on $X$ satisfies 
$$
 \mathcal{L}_{X, \mathcal{A}} (f) = \sup_{x, x' \in X \atop x \neq x'} \frac{\| f(x) - f(x') \|_{\mathcal{A}}}{d_X (x, x')}
 < \infty ,
$$
then we say that $f$ is a \textsl{Lipschitz function}.
By $\Lip (X, \mathcal{A})$, 
we denote the set of all $\mathcal{A}$-valued Lipschitz functions on $X$.
Clearly, $\Lip (X, \mathcal{A}) \subset C(X, \mathcal{A})$ and 
$\Lip (X, \mathcal{A})$ is a Banach algebra with norm
$$
 \| f \|_{\Lip (X, \mathcal{A})} = \| f \|_{C(X, \mathcal{A})} + \mathcal{L}_{X, \mathcal{A}} (f) .
$$
In case that $\mathcal{A} = \mathbb{C}$, we write $C(X) = C(X, \mathbb{C})$ and $\Lip (X) = \Lip (X, \mathbb{C})$.
The Lipschitz algebra $\Lip (X)$ has been well studied.
The researches on this subject may be found in the book \cite{W}.
Here a mapping between two Banach algebras is said to be a \textsl{homomorphism}
if it preserving addition, scalar multiplication and multiplication.
Moreover if it maps unit to unit, then we say that it is \textsl{unital}.
One of the results is the description of homomorphisms between Lipschitz algebras.

\begin{thmA}[Sherbert, {\cite[Proposition 2.1]{Sh}}]
Suppose that $X$ and $Y$ are compact metric spaces with metrics $d_X$ and $d_Y$ respectively.
Then $T$ is a unital homomorphism from $\Lip (X)$ into $\Lip (Y)$, if and only if there exists a mapping $\varphi : Y \rightarrow X$ with 
$$
 \sup_{y, y' \in Y \atop y \neq y'} \frac{d_X ( \varphi (y), \varphi (y') )}{d_Y (y, y')} < \infty, 
$$
such that 
$$
 (Tf)(y) = f( \varphi (y) ) \qquad ( y \in Y )
$$
for all $f \in \Lip (X)$.
\end{thmA}

This theorem has been developed in several directions.
In \cite{BJ}, 
F. Botelho and J. Jamison replaced $\Lip (X)$ by $\Lip (X, \mathcal{A})$, where 
$\mathcal{A}$ is the Banach algebra $\mathbf{c}$ of convergent sequences or 
the Banach algebra $\ell^{\infty}$ of bounded sequences.
They determined the unital homomorphisms from $\Lip (X, \mathbf{c})$ into $\Lip (Y, \mathbf{c})$ 
and those from $\Lip (X, \ell^{\infty})$ into $\Lip (Y, \ell^{\infty})$, 
where $X$ and $Y$ are compact metric spaces.

In general, if $K$ is a compact Hausdorff space, 
then $C(K)$ denotes the Banach algebras of all complex-valued continuous functions on $K$, 
with norm $\| f \|_{C(K)} = \sup_{\xi \in K} |f( \xi )|$.
In \cite{Oi}, S. Oi took up the algebra $\Lip (X, C(K))$ and proved the following theorem:

\begin{thmA}[{\cite[Oi]{Oi}}]\label{ThmOi}
Suppose that $X$ and $Y$ are as in Theorem A, and that $K$ and $M$ are compact Hausdorff spaces.
Assume that $Y$ is connected.
Then $T$ is a unital homomorphism from $\Lip (X, C(K))$ into $\Lip (Y, C(M))$
if and only if there exist a class $\{ \varphi_{\eta} \}_{\eta \in M}$ of mappings from $Y$ to $X$ with the properties {\rm (a)} and {\rm (b)}
and a continuous mapping $\psi : M \rightarrow K$ such that 
$$
 \bigl[ (Tf)(y) \bigr] (\eta) = \bigl[ f( \varphi_{\eta} (y) ) \bigr] ( \psi (\eta) ) 
 \qquad ( y \in Y, \ \eta \in M ) 
$$
for all $f \in \Lip(X, C(K))$.
\begin{enumerate}
\renewcommand{\labelenumi}{(\alph{enumi})}
\item
For each $y \in Y$, the mapping $\eta \mapsto \varphi_{\eta}(y)$ from $M$ to $X$ is continuous. 
\item
$\displaystyle \sup_{\eta \in M} \sup_{y, y' \in Y \atop y \neq y'} \frac{d_X ( \varphi_{\eta} (y), \varphi_{\eta} (y') )}{d_Y(y, y')} < \infty$.
\end{enumerate}
\end{thmA}

This theorem leads to the result of Botelho and Jamison mentioned above.
Here we turn our attention to two assumptions in Theorem \ref{ThmOi}.
One %assumption 
is that $Y$ is connected and the other is that $T$ is unital.
These assumptions seem to be inessential but they simplify the statement of theorem.
In order to remove these assumptions and to state a general result, 
we consider a function $f$ in $\Lip (X, C(K))$ 
as a function of two variables $x \in X$ and $\xi \in K$.
So we write $f(x, \xi)$ instead of $\bigl[ f(x) \bigr] (\xi)$.
Let $f$ be a function on $X \times K$.
With $x \in X$ we associate a function $f_x$ defined on $K$ by $f_x (\xi) = f(x, \xi)$. 
Similarly, if $\xi \in K$, $f^{\xi}$ is the function defined on $X$ by $f^{\xi} (x) = f(x, \xi)$.
In general, for any mapping of two variables, we use the same expression:
For example, if $\psi : Y \times M \rightarrow K$, 
then $\psi^{\eta} : Y \rightarrow K$ and $\psi_y : M \rightarrow K$ are 
defined by $\psi^{\eta} (y) = \psi (y, \eta)$ and $\psi_y (\eta) = \psi (y, \eta)$.

A subset $A$ of a topological space is said to be \textsl{clopen}, 
if $A$ is both open and closed.
We do not exclude the possibility that a clopen set is empty.
We understand that the statement about an empty set is true.

\begin{thm}\label{Thm1}
Suppose that $X$ and $Y$ are compact metric spaces with metrics $d_X$ and $d_Y$ respectively, 
and that $K$ and $M$ are compact Hausdorff spaces.
If $T$ is a homomorphism from $\Lip (X, C(K))$ into $\Lip (Y, C(M))$, 
then there exist a clopen subset $\mathcal{D}$ of $Y \times M$ and 
two continuous mappings $\varphi : \mathcal{D} \rightarrow X$ and 
$\psi : \mathcal{D} \rightarrow K$ with {\rm (i)} and {\rm (ii)} 
such that $T$ has the form:
\begin{equation}\label{eq:Form}
 (Tf) (y, \eta) = \begin{cases}
  f \bigl( \varphi (y, \eta), \psi(y, \eta) \bigr)  &  \bigr( (y, \eta) \in \mathcal{D} \bigr) \\
 0  &  \bigl( (y, \eta) \in ( Y \times M ) \smallsetminus \mathcal{D} \bigr)
 \end{cases} 
\end{equation}
for all $f \in \Lip(X, C(K))$.
\begin{enumerate}
\renewcommand{\labelenumi}{(\roman{enumi})}
\item
There exists a bound $L \geq 0$ such that 
\begin{equation}\label{eq:Phi}
 (y, \eta), (y', \eta) \in \mathcal{D} \ \mbox{and} \ y \neq y' 
 \ \mbox{imply} \ \
 \frac{d_X ( \varphi (y, \eta), \varphi (y',\eta) )}{d_Y(y, y')}  \leq  L.
\end{equation}
\item
For any $\eta \in M$, the set $\mathcal{D}^{\eta} = \{ y \in Y : (y, \eta) \in \mathcal{D} \}$
is a union of finitely many disjoint clopen subsets $V^{\eta}_1, \ldots , V^{\eta}_{n_{\eta}}$ of $Y$ such that 
$$
 \mbox{$\psi^{\eta}$ is constant on $V^{\eta}_i$ \ for $i=1, \ldots , n_{\eta}$,}
$$
and 
\begin{equation}\label{eq:SetDist}
 d_Y ( V^{\eta}_i, V^{\eta}_j ) \geq r \qquad ( i \neq j ) .
\end{equation}
Here $r$ is a positive constant independent of $\eta$.
\end{enumerate}
Conversely, if $\mathcal{D}$, $\varphi$, $\psi$ are given as above, 
then $T$ defined by {\rm (\ref{eq:Form})} is a homomorphism from $\Lip (X, C(K))$ into $\Lip (Y, C(M))$.
Moreover, $T$ is unital if and only if $\mathcal{D} = Y \times M$.
\end{thm}

In (\ref{eq:SetDist}), $d_Y (A, B)$ denotes the usual distance between two sets $A, B \subset Y$,
that is, $d_Y (A, B) = \inf \{ d_Y (y, y') : y \in A, y' \in B \}$.
(If $A = \emptyset$ or if $B = \emptyset$, then we set 
$d_Y (A, B) = \infty$).

\bigskip

Next we consider the following problem: 
$$
 \mbox{When is a homomorphism between Lipschitz algebras compact?}
$$
In \cite{KS}, H. Kamowitz and S. Scheinberg answered to this problem as follows:

\begin{thmA}[Kamowitz and Scheinberg, {\cite{KS}}]
Let $T$ be a unital homomorphism from $\Lip (X)$ into $\Lip (Y)$ 
described in Theorem A.
Then $T$ is compact if and only if 
$$
 \lim_{d_Y (y, y') \to 0} \frac{d_X ( \varphi (y), \varphi (y') )}{d_Y (y, y')}  = 0.
$$
\end{thmA}

%In \cite{}, Botelho and Jamison dealt with 
%a homomorphism $T$ from $\Lip (X, \mathbb{c})$ into 
%$\Lip (Y, \mathbf{c})$ or from $\Lip (x, \ell^{\infty}$ into $\Lip (Y, \ell^{\infty})$.
%They gave the sufficient condition for $T$ to be compact.
%But the condition is not necessary.
In this paper, we give a necessary and sufficient condition 
for $T$ in Theorem \ref{Thm1} to be compact.

\begin{thm}\label{Thm2}
Let $X$, $Y$, $K$, $M$ be as in Theorem \ref{Thm1}.
Suppose that $T$ is a homomorphism from $\Lip (X, C(K))$ into $\Lip (Y, C(M))$ 
with form {\rm (\ref{eq:Form})} in Theorem \ref{Thm1}. 
Then $T$ is compact if and only if {\rm (iii)} and {\rm (iv)} hold.
\begin{enumerate}
\renewcommand{\labelenumi}{(\roman{enumi})}
\setcounter{enumi}{2}
\item
For any $\varepsilon > 0$, 
there exists $\delta > 0$ such that 
\begin{equation}\label{}
 (y, \eta), (y', \eta) \in \mathcal{D} \ \mbox{and} \ 
 0 < d_Y (y, y') < \delta \ \mbox{imply} \ \  
 \frac{d_X ( \varphi (y, \eta), \varphi (y', \eta) )}{d_Y(y, y')} 
 < \varepsilon.
\end{equation}
\item
For any $y \in Y$, the set $\mathcal{D}_y = \{ \eta \in M : (y, \eta) \in \mathcal{D} \}$
is a union of finitely many disjoint clopen sets $\Omega_y^1, \ldots , \Omega_y ^{n_y}$ such that 
$$
 \mbox{$\psi_y$ is constant on $\Omega_y^i$
 \ for $i=1, \ldots , n_y$}.
$$
\end{enumerate}
\end{thm}

\section{Preliminaries}%第2章

As mentioned in Introduction, 
we consider a function $f$ in $\Lip (X, C(K))$ as a function on $X \times K$.

\begin{prop}\label{Prop:LipFn}
Let $f$ be a complex-valued function on $X \times K$.
Then $f \in \Lip (X, C(K))$ if and only if 
$f \in C(X \times K)$ and 
\begin{equation}\label{eq:LipFn}
 \mathcal{L}_{X, C(K)} (f) = 
 \sup _{x, x' \in X \atop x \neq x'} \frac{\| f_x - f_{x'} \|_{C(K)}}{d_X (x, x')} < \infty.
\end{equation}
Moreover, $\| f \|_{C(X, C(K))} = \| f \|_{C(X \times K)}$.
\end{prop}

\begin{proof}
Straightforward.
\end{proof}

The next proposition implies that $\Lip (X, C(K))$ separates the points of $X \times K$.

\begin{prop}\label{Prop:LipSep}
For any $(x_0, \xi_0) \in X \times K$ and for any open neighborhood $\mathcal{U}$ of $(x_0, \xi_0)$, 
there exists a $f \in \Lip (X, C(K))$ such that 
$0 \leq f \leq 1$, $f( x_0, \xi_0)= 1$ and 
$f(x, \xi) \leq m<1$ for all $(x, \xi ) \in (X \times K) \smallsetminus \mathcal{U}$.
\end{prop}

\begin{proof}
Let $(x_0, \xi_0) \in X \times K$ and let $\mathcal{U}$ be an 
open neighborhood of $(x_0, \xi_0)$.
Then there exist an open neighborhood $U$ of $x_0$ in $X$ and an open neighborhood $\Theta$ of $\xi_0$ in $K$ such that $(x_0, \xi_0) \in U \times \Theta \subset \mathcal{U}$.

Let $h$ be a function on $X$ defined by 
$$
 h(x) = 1 - \frac{d_X (x, x_0)}{\operatorname{diam} (X)} \qquad (x \in X),
$$
where $\operatorname{diam} (X) = \sup \{ d_X (x, x') : x, x' \in X \}$.
We easily see that $h \in \Lip (X)$, 
$0 \leq h \leq 1$, $h(x_0) = 1$ and $h(x) < 1$ for all $x \in X \smallsetminus \{ x_0 \}$.
By Urysohn's lemma, there is a $u \in C(K)$ such that 
$0 \leq u \leq 1$, $u( \xi_0 ) = 1$ and $u( \xi ) = 0$ for all $\xi \in K \smallsetminus \Theta$.
Now, put $f(x, \xi) = h(x) u(\xi)$ for $(x, \xi) \in X \times K$.
Then we can verify that $f$ has the desired properties.
Here $m$ may be taken as the maximum of $f$ on the compact set $( X \times K ) \smallsetminus \mathcal{U}$.
\end{proof}

Here we summarize a fundamental fact on the Banach algebra $\Lip (X, C(K))$.

\begin{prop}\label{Prop:MIS}
$\Lip (X, C(K))$ is a semi-simple unital commutative Banach algebra and 
its maximal ideal space is identified with $X \times K$.
In fact, for any multiplicative linear functional $\Psi$ on $\Lip (X, C(K))$, 
there exists a unique point $(x, \xi) \in X \times K$ such that 
$\Psi (f) = f(x, \xi)$ for all $f \in \Lip (X, C(K))$.
\end{prop}

We can prove this proposition by the well-known argument in theory of Banach algebras.
%For instance, 
The details may be found in \cite[Propositions 11 and 12]{Oi}.

\section{Proof of Theorem \ref{Thm1}}%第3章

In this section we prove Theorem \ref{Thm1}.

\subsection{{\sc Proof of Sufficiency}}

We first settle the converse statement.
Suppose that $\mathcal{D}$ is a clopen subset of $Y \times M$, 
that $\varphi : \mathcal{D} \rightarrow X$ and $\psi : \mathcal{D} \rightarrow K$ are continuous mappings with (i) and (ii), 
and that $T$ is defined by (\ref{eq:Form}).

\begin{lem}\label{L1}
If 
$$
 \rho = 
\inf \bigl\{ d_Y (y, y') \: : \:
 \mbox{$(y, \eta) \in \mathcal{D}$ and $(y', \eta) \in (Y \times M) \smallsetminus \mathcal{D}$ for some $\eta \in M$} \bigr\}, 
$$
then $\rho > 0$.
\end{lem}

If there is no pair $(y, y') \in Y \times Y$ 
such that $(y, \eta ) \in \mathcal{D}$ and $(y', \eta) \in (Y \times M) \smallsetminus \mathcal{D}$ for some $\eta \in K$, 
then we understand that $\rho = \infty$.

\begin{proof}
Conversely, assume that $\rho = 0$.
Then for each $n=1, 2, \ldots$, 
there exist $(y_n, \eta_n) \in \mathcal{D}$ and 
$(y'_n, \eta_n) \in (Y \times M) \smallsetminus \mathcal{D}$
such that $d_Y (y_n, y'_n) < 1/n$.
Since $\mathcal{D}$ is compact, 
there exist a net $\{ n_{\alpha} \}$ and a point $(y, \eta) \in \mathcal{D}$ such that $y_{n_{\alpha}} \rightarrow y$ and $\eta_{n_{\alpha}} \rightarrow \eta$.
Then $d_Y (y_{n_{\alpha}}, y'_{n_{\alpha}}) < 1/n_{\alpha} \rightarrow 0$.
Hence $y'_{n_{\alpha}} \rightarrow y$ and so $(y'_{n_{\alpha}}, \eta_{n_{\alpha}}) \rightarrow (y, \eta)$.
Since $(Y \times M) \smallsetminus \mathcal{D}$ is closed, 
we get $(y, \eta) \in (Y \times M) \smallsetminus \mathcal{D}$.
%This is a contradiction, because $(y, \eta) \in \mathcal{D}$.
This contradicts the fact that $(y, \eta) \in \mathcal{D}$.
Consequently, we have $\rho > 0$.
\end{proof}

\begin{lem}\label{L2}
For any $f \in \Lip (X, C(K))$, $Tf \in \Lip (Y, C(M))$.
\end{lem}

\begin{proof}
Let $f \in \Lip (X, C(K))$.
By Proposition \ref{Prop:LipFn}, we have 
$f \in C( X \times K )$ and (\ref{eq:LipFn}).
 
We first show that $Tf \in C( Y \times M )$.
Since $\varphi : \mathcal{D} \rightarrow X$ and $\psi : \mathcal{D} \rightarrow K$ are continuous and since $f \in C( X \times K )$, 
the first line in (\ref{eq:Form}) implies that $Tf$ is continuous on $\mathcal{D}$.
Of course, the second one implies that it is so on $( Y \times M ) \smallsetminus \mathcal{D}$.
Noting that $\mathcal{D}$ is clopen, we see that $Tf$ is continuous on $Y \times M$.

To see that $Tf \in \Lip (Y, C(M))$, 
it suffices to show that 
\begin{equation}\label{eq:L2-2}
 \mathcal{L}_{Y, C(M)} (Tf)
 = \sup _{y, y' \in Y \atop y \neq y'}
 \frac{\| (Tf)_y - (Tf)_{y'} \|_{C(M)}}{d_Y (y,y')} < \infty .
\end{equation}
For this end, choose $y, y' \in Y$ so that $y \neq y'$ 
and let $\eta \in M$.
We consider three cases.

[\textsf{Case 1}] $(y, \eta), (y', \eta) \in \mathcal{D}$ :
By (ii), $y \in V^{\eta}_i$ and $y' \in V^{\eta}_{i'}$ for some 
$i, i' \in \{ 1, \ldots, n_{\eta} \}$.
We first consider the case $i = i'$.
Then $y, y' \in V^{\eta}_i$.
Since $\psi^{\eta}$ is constant on $V^{\eta}_i$, 
$\psi (y, \eta) = \psi^{\eta} (y) = \psi^{\eta} (y') = \psi (y', \eta)$.
Put $x = \varphi (y, \eta)$, $x' = \varphi(y', \eta)$ 
and $\xi = \psi (y, \eta) = \psi (y', \eta)$.
Using (\ref{eq:Form}), we compute 
\begin{equation}\label{eq:L2-3}\begin{split}
 \bigl| (Tf)(y, \eta) - (Tf)(y', \eta) \bigr| 
 & = \bigl| f( \varphi (y, \eta), \psi (y, \eta) ) 
          - f( \varphi (y', \eta), \psi (y', \eta) ) \bigr| \\
 & = | f(x, \xi) - f(x', \xi) | 
   = | f_x (\xi) - f_{x'} (\xi) | \\
 & \leq \| f_x - f_{x'} \|_{C(K)} \\
 & \leq \mathcal{L}_{X, C(K)} (f) \, d_X (x,x') \\
 & = \mathcal{L}_{X, C(K)} (f) \: 
     d_X \bigl( \varphi (y, \eta), \varphi (y', \eta) \bigr) \\
 & \leq \mathcal{L}_{X, C(K)} (f) \: L \: 
     d_Y (y, y'),
\end{split}\end{equation}
where the fourth and last lines 
follow from (\ref{eq:LipFn}) and (\ref{eq:Phi}), respectively.

On the other hand, if $i \neq i'$, 
then (\ref{eq:SetDist}) yields $d_Y (y, y') \geq d_Y (V^{\eta}_i, V^{\eta}_{i'}) \geq r$.
Hence 
\begin{equation}\label{eq:L2-4}
 \frac{\bigl| (Tf)(y, \eta) - (Tf)(y', \eta) \bigr|}{d_Y (y, y')} 
 \leq  \frac{\bigl| (Tf)(y, \eta) \bigr| + \bigl| (Tf)(y', \eta) \bigr|}{r}
 \leq \frac{2\, \| f \|_{C(X \times K)}}{r} .
\end{equation}

[\textsf{Case 2}] $(y, \eta) \in \mathcal{D}$ and $(y', \eta) \in (Y \times M) \smallsetminus \mathcal{D}$ :
Then Lemma \ref{L1} says that 
$d_Y (y, y') \geq \rho > 0$.
By (\ref{eq:Form}), we get  
\begin{equation}\label{eq:L2-5}
 \frac{\bigl| (Tf)(y, \eta) - (Tf)(y', \eta) \bigr|}{d_Y (y, y')} 
 \leq \frac{\bigl| f( \varphi (y, \eta), \psi (y, \eta) ) - 0 \bigr|}{\rho} 
 \leq \frac{\| f \|_{C(X \times K)}}{\rho} .
\end{equation}

[\textsf{Case 3}] $(y, \eta), (y', \eta) \in (Y \times M) \smallsetminus \mathcal{D}$ :
By (\ref{eq:Form}), 
\begin{equation}\label{eq:L2-6}
 (Tf)(y, \eta) - (Tf)(y', \eta) =0.
\end{equation}

Combining (\ref{eq:L2-3})--(\ref{eq:L2-6}), 
we can arrive at (\ref{eq:L2-2}).
Indeed, if we put $C = \max \bigl\{ L, 2/r, 1/\rho \bigr\}$, then we have 
\begin{equation}\label{eq:L2-7}
 \mathcal{L}_{Y, C(M)} (Tf) 
  =  \sup_{y, y' \in Y \atop y \neq y'} \sup_{\eta \in M} \frac{\bigl| (Tf)(y, \eta) - (Tf)(y', \eta) \bigr|}{d_Y (y, y')}  
   \leq C \, \| f \|_{\Lip (X, C(K))}, 
\end{equation}
because $\mathcal{L}_{X, C(K)} (f) \leq \| f \|_{\Lip (X, C(K))}$ and $\| f \|_{C(X \times K)} \leq \| f \|_{\Lip (X, C(K))}$.
\end{proof}

Lemma \ref{L2} says that $T$ maps $\Lip (X, C(K))$ into $\Lip (Y, C(M))$.
While the form (\ref{eq:Form}) shows that $T$ is a homomorphism.
Thus we obtain the converse statement of Theorem \ref{Thm1}.

\begin{rem}
From (\ref{eq:Form}), we see that $\| Tf \|_{C(Y \times M)} \leq \| f \|_{C(X \times K)}$.
Using this and (\ref{eq:L2-7}),  we obtain the norm estimate  
$$
 \| T \| = \sup_{\| f \|_{\Lip (X, C(K))} \leq 1} \| Tf \|_{\Lip (Y, C(M))}  \leq  C+1.
$$
This estimate is not sharp, 
but it seems to be difficult to give an exact expression of $\| T \|$.
\end{rem}

\subsection{{\sc Proof of Necessity}}%

We turn to the proof of the main statement of Theorem \ref{Thm1}.
Suppose that $T$ is an arbitrary homomorphism from $\Lip (X, C(K))$ into $\Lip (Y, C(M))$.
Since $\Lip (Y, C(M))$ is semi-simple, we know from \cite[Theorem 11.10]{R} that $T$ is continuous.
Thus the norm $\| T \|$ is determined as a bounded linear operator $T$.

If $T=O$, then we only take $\mathcal{D} = \emptyset$.
So, we assume that $T \neq O$.

\begin{lem}\label{L3}
There exist a clopen subset $\mathcal{D}$ of $Y \times M$ 
and two mapping $\varphi : \mathcal{D} \rightarrow X$ and 
$\psi : \mathcal{D} \rightarrow K$ such that {\rm (\ref{eq:Form})} holds.
\end{lem}

\begin{proof}
Let $\mathbf{1}$ denote the unit of $\Lip (X, C(K))$, 
namely, the constant $1$ function on $X \times K$.
Since $(T \mathbf{1})^2 = T( \mathbf{1}^2 ) = T \mathbf{1}$, 
we have $(T \mathbf{1})(y, \eta) \in \{ 1, 0 \}$ for all $(y, \eta) \in Y \times M$.
Put 
\begin{equation}\label{eq:L3-1}
 \mathcal{D} = \{ (y, \eta) \in Y \times M : (T\mathbf{1})(y, \eta) = 1 \}.
\end{equation}
Then 
$$
 (Y \times M) \smallsetminus \mathcal{D} 
 = \{ (y, \eta ) \in Y \times M : (T\mathbf{1})(y, \eta) = 0 \} .
$$
Since $T\mathbf{1}$ is continuous on $Y \times M$, 
both $\mathcal{D}$ and $(Y \times M) \smallsetminus \mathcal{D}$ are closed.
Hence $\mathcal{D}$ is clopen.

To determine the mappings $\varphi : \mathcal{D} \rightarrow X$ and 
$\psi : \mathcal{D} \rightarrow K$, fix any $(y, \eta) \in \mathcal{D}$.
Define a functional $\Psi_{(y, \eta)}$ on $\Lip (X, C(K))$ by 
$$
 \Psi_{(y, \eta)} (f) = (Tf)(y, \eta) \qquad \bigl( f \in \Lip (X, C(K)) \bigr).
$$
Then $\Psi_{(y, \eta)}$ is a homomorphism from $\Lip (X, C(K))$ into $\mathbb{C}$.
Moreover, (\ref{eq:L3-1}) yields $\Psi_{(y, \eta)} (\textbf{1}) = (T\textbf{1})(y, \eta) = 1$.
Hence $\Psi_{(y, \eta)}$ is a multiplicative linear functional on $\Lip (X, C(K))$.
Thus Proposition \ref{Prop:MIS} gives 
a unique point $(x, \xi) \in X \times K$ such that 
$$
 \Psi_{(y, \eta)} (f) = f(x, \xi) \qquad \bigl( f \in \Lip (X, C(K)) \bigr).
$$
By putting $\varphi (y, \eta) = x$ and $\psi (y, \eta) = \xi$, 
we determine the mappings $\varphi : \mathcal{D} \rightarrow X$ 
and $\psi : \mathcal{D} \rightarrow K$.
Then, for any $f \in \Lip (X, C(K))$, 
\begin{equation}\label{eq:L3-2}
 (Tf)(y, \eta) = \Psi_{(y, \eta)} (f) = f(x, \xi) = f \bigl( \varphi(y, \eta), \psi (y, \eta) \bigr) .
\end{equation}

Finally, if $(y, \eta) \in (Y \times M) \smallsetminus \mathcal{D}$, 
then $(T\mathbf{1})(y, \eta) = 0$ and so 
for any $f \in \Lip (X, C(K))$, 
$Tf = T(f \mathbf{1}) = (Tf)(T\mathbf{1})$ and so 
$$
 (Tf)(y, \eta) = (Tf)(y, \eta) \, (T\mathbf{1})(y, \eta) = 0.
$$
Together with (\ref{eq:L3-2}), we establish (\ref{eq:Form}).
\end{proof}

\begin{lem}\label{L4}
The mappings $\varphi : \mathcal{D} \rightarrow X$ and $\psi : \mathcal{D} \rightarrow K$ are continuous.
\end{lem}

\begin{proof}
Define a mapping $\Phi : \mathcal{D} \rightarrow X \times K$ by 
$$
 \Phi (y, \eta) = \bigl( \varphi (y, \eta), \psi (y, \eta) \bigr) \qquad \bigl( (y, \eta) \in \mathcal{D} \bigr) .
$$
We prove the lemma by verifying that $\Phi$ is continuous at each point $( y_0, \eta_0 ) \in \mathcal{D}$.
Let $\mathcal{U}$ be an arbitrary open neighborhood of $\Phi ( y_0, \eta_0 )$ in $X \times K$. 
By Proposition~\ref{Prop:LipSep}, 
there exists an $f \in \Lip (X, C(K))$ such that $0 \leq f \leq 1$,
 $f( \Phi (y_0,  \eta_0) ) = 1$ and 
\begin{equation}\label{eq:L4-1}
  0 \leq f(x, \xi) \leq m < 1  \qquad  
 \bigl( (x, \xi) \in ( X \times K ) \smallsetminus \mathcal{U} \bigr).
\end{equation}
Put $\mathcal{V} = \bigl\{ (y, \eta) \in \mathcal{D} : |(Tf)(y, \eta)| > m \bigr\}$.
Since $Tf$ is continuous on $Y \times M$, $\mathcal{V}$ is open.
Also, $( y_0, \eta_0 ) \in \mathcal{V}$ because
$$
 (Tf)( y_0, \eta_0 ) = f\bigl( \varphi( y_0, \eta_0 ), \psi ( y_0, \eta_0 ) \bigr) = f \bigl( \Phi ( y_0, \eta_0 ) \bigr) = 1 > m.
$$
Moreover, if $(y, \eta) \in \mathcal{V}$, then 
$$
 \bigl| f( \Phi (y, \eta) ) \bigr| 
 = \bigl| f \bigl( \varphi (y, \eta), \psi (y, \eta) \bigr) \bigr|
 = \bigl| (Tf)(y, \eta) \bigr| > m
$$
and (\ref{eq:L4-1}) forces that $\Phi (y, \eta) \in \mathcal{U}$.
Hence $\Phi ( \mathcal{V} ) \subset \mathcal{U}$.
Thus $\Phi$ is continuous at $(y_0, \eta_0)$, as desired.
\end{proof}

\begin{lem}
$\varphi$ satisfies {\rm (i)}.
\end{lem}

\begin {proof}
Let $(y, \eta), (y', \eta) \in \mathcal{D}$ with $y \neq y'$.
Put $x_0 = \varphi (y', \eta)$ and
$$
 f(x) = d_X (x, x_0)  \qquad (x \in X).
$$
Then $f \in \Lip (X)$ and $\| f \|_{\Lip (X)} \leq \operatorname{diam} (X) + 1$.
Extend $f$ to $X \times K$ by $\hat{f} (x, \xi) = f(x)$ for all $(x, \xi) \in X \times K$.
Clearly $\hat{f} \in \Lip (X, C(K))$ and $\| \hat{f} \|_{\Lip (X, C(K))} = \| f \|_{\Lip (X)}$.
Moreover, we have
\begin{equation*}\begin{split}
 d_X  \bigl( \varphi (y, \eta), \varphi ( y', \eta ) \bigr) & 
 = \bigl| d_X \bigl( \varphi (y, \eta), x_0 \bigr) - d_X \bigl( \varphi (y', \eta), x_0 \bigr) \bigr| \\
 & = \bigl| f ( \varphi (y, \eta) ) - f( \varphi (y', \eta) ) \bigr| \\
 & = \bigl| \hat{f} \bigl( \varphi (y, \eta), \psi (y, \eta) \bigr) 
                    - \hat{f} \bigl( \varphi (y', \eta), \psi (y', \eta) \bigr) \bigr| \\
 & = \bigl| (T\hat{f}) (y, \eta) - (T\hat{f}) (y', \eta) \bigr|
    = \bigl| (T\hat{f})_y (\eta) - (T\hat{f})_{y'} (\eta) \bigr| \\
 & \leq \bigl\| (T\hat{f})_y - (T\hat{f})_{y'} \bigr\|_{C(M)} \\
 & \leq \mathcal{L}_{Y, C(M)} (T\hat{f}) \: d_Y (y, y').
\end{split}\end{equation*}
Since $\mathcal{L}_{(Y, C(M))} (T\hat{f}) \leq \| T\hat{f} \|_{\Lip (Y, C(M))} \leq \| T \| \, \| \hat{f} \|_{\Lip(X, C(K))} 
\leq \| T \|  ( \operatorname{diam} (X) + 1 )$, we obtain
$$
 \frac{d_X  \bigl( \varphi (y, \eta), \varphi ( y', \eta ) \bigr) }{d_Y (y, y')} 
    \leq \| T \| \; ( \operatorname{diam} (X) + 1 ), 
$$
which is (i).
\end{proof}

\begin{lem}\label{L6}
There exists an $r>0$ such that 
$$
 (y, \eta), (y', \eta) \in \mathcal{D} \ \mbox{and} \ \ d_Y (y, y') < r
 \ \mbox{imply} \ \ \psi^{\eta}  (y) = \psi^{\eta} (y') .
$$
\end{lem}

\begin{proof}
Take $r$ so that $0 < r < 1/ \| T \|$.
Choose $(y, \eta), (y', \eta) \in \mathcal{D}$ with $d_Y (y, y') < r$ and assume that $\psi^{\eta} (y) \neq \psi ^{\eta} (y')$.
By Urysohn's lemma, we find a $u \in C(K)$ such that $0 \leq u \leq 1$, $u (\psi^{\eta} (y)) = 1$ and 
$u (\psi^{\eta} (y')) = 0$.
Define a function on $X \times K$ as $\tilde{u} (x, \xi) = u(\xi)$ for all $(x, \xi) \in X \times K$. 
Then $\tilde{u} \in \Lip (X, C(K))$ and $\| \tilde{u} \|_{\Lip(X, C(K))} = \| u \|_{C(K)} = 1$.
Moreover we have 
\begin{equation*}\begin{split}
 1 & 
 = \bigl| u(\psi^{\eta} (y)) - u (\psi^{\eta} (y')) \bigr|
 = \bigl| u (\psi (y, \eta)) - u (\psi (y', \eta)) \bigr|  \\ 
 & = \bigl| \tilde{u} \bigl( \varphi (y, \eta), \psi (y, \eta) \bigr) 
                   - \tilde{u} \bigl( \varphi (y', \eta), \psi (y',\eta) \bigr) \bigr| \\
 & = \bigl| (T\tilde{u}) (y, \eta) - (T\tilde{u}) (y', \eta) \bigr| \\
 & \leq \bigl\| (T\tilde{u})_y - (T\tilde{u})_{y'} \bigr\|_{C(M)} \\
 & \leq \mathcal{L}_{Y, C(M)} (T\tilde{u}) \; d_Y (y, y') \\
 & <  \| T\tilde{u} \|_{\Lip (Y, C(M))} \, r 
   \leq \| T \| \, \| \tilde{u} \|_{\Lip (X, C(K))} \, r 
 = \| T \| \, r < 1,
 \end{split}\end{equation*}
a contradiction.
Hence $\psi^{\eta} (y)= \psi^{\eta} (y')$.
\end{proof}

\begin{lem}
$\psi$ satisfies {\rm (ii)}.
\end{lem}

\begin {proof}
Fix any $\eta \in M$ and put $\mathcal{D}^{\eta} = \{ y \in Y : (y, \eta) \in \mathcal{D} \}$.
Since $\mathcal{D}$ is clopen, 
$\mathcal{D}^{\eta}$ is a clopen subset of $Y$.
For any $y \in \mathcal{D}^{\eta}$, put 
\begin{equation}\label{eq:L7-1}
 V_y = \{ z \in \mathcal{D}^{\eta} : \psi^{\eta} (z) = \psi^{\eta} (y) \}.
\end{equation}
Clearly, $\psi^{\eta}$ is constant on $V_y$.
Also, we have 
\begin{equation}\label{eq:L7-2}
 V_y \cap V_{y'} \neq \emptyset \ \Longrightarrow \ V_y = V_{y'}.
\end{equation}

Since $\psi^{\eta}$ is continuous by Lemma \ref{L4}, 
$V_y$ is a closed subset of $\mathcal{D}^{\eta}$.
To see that $V_y$ is an open subset of $\mathcal{D}^{\eta}$, 
let $z \in V_y$ and consider 
an $r$-ball $B(z; r) = \{ w \in \mathcal{D}^{\eta} : d_Y (w, z) < r \}$, 
where $r$ is given in Lemma \ref{L6}.
If $w \in B(z; r)$, then $(w, \eta), (z, \eta) \in \mathcal{D}$ and $d_Y (w, z) < r$.
Hence Lemma \ref{L6} implies that $\psi^{\eta} (w) = \psi^{\eta} (z) = \psi^{\eta} (y)$, 
and so $w \in V_y$.
Therefore $B(z; r) \subset V_y$.
Thus $V_y$ is an open subset of $\mathcal{D}^{\eta}$.
Consequently, $V_y$ is a clopen subset of $Y$.

Note that 
$$
 \mathcal{D}^{\eta} = \bigcup_{y \in \mathcal{D}^{\eta}} V_y .
$$
Since $\mathcal{D}^{\eta}$ is compact, 
we can select finitely many 
$y_1, \ldots , y_n \in \mathcal{D}^{\eta}$ such that 
$$
 \mathcal{D}^{\eta} = \bigcup_{i=1}^n V_{y_i} .
$$
By (\ref{eq:L7-2}), 
we may assume that $V_{y_1}, \ldots , V_{y_n}$ are disjoint.

Finally we show that $d_Y ( V_{y_i}, V_{y_j} ) \geq r$ ($i \neq j$).
Assume that $d_Y ( V_{y_i}, V_{y_j} ) < r$.
Then there exist $z_i \in V_{y_i}$ and $z_j \in V_{y_j}$ such that 
$d_Y (z_i, z_j) < r$.
By Lemma \ref{L6}, $\psi^{\eta} (z_i) = 
\psi^{\eta} (z_j)$, and hence (\ref{eq:L7-1}) and (\ref{eq:L7-2}) yield $V_{y_i} = V_{y_j}$.
Since $V_{y_1}, \ldots , V_{y_n}$ is disjoint, 
we must have $d_Y ( V_{y_i}, V_{y_j} ) \geq r$ ($i \neq j$).

Putting $n_{\eta} = n$ and writing $V^{\eta} _i$ for $V_{y_i}$ ($i = 1, \ldots , n_{\eta}$), we obtain (ii).
\end{proof}

Thus the proof of Theorem \ref{Thm1} is completed.

\section{Proof of Theorem \ref{Thm2}}%第4章

In this section, we prove Theorem \ref{Thm2}.
Throughout this section, $T$ is a homomorphism from $\Lip (X, C(K))$ into $\Lip (Y, C(M))$
 with the form (\ref{eq:Form}) in Theorem \ref{Thm1}.
Of course, the set $\mathcal{D}$ and the mappings $\varphi$ and $\psi$ are as in Theorem \ref{Thm1}.
Since $T$ is bounded, we use its norm $\| T \|$ again.
Let $\mathbb{B}_{\Lip (X, C(K))}$ 
be the unit ball of $\Lip (X, C(K))$, that is,
$$
 \mathbb{B}_{\Lip (X, C(K))} 
 = \bigl\{ f \in \Lip (X, C(K)) : \| f \|_{\Lip (X, C(K))} \leq 1 \bigr\} .
$$

\subsection{{\sc Proof of Sufficiency}}

We first show the ``if'' part in Theorem \ref{Thm2}.

Suppose that $\varphi$ and $\psi$ satisfy (iii) and (iv) respectively.
We prove that $T$ is compact.
Here we may assume that $T \neq O$, 
otherwise there is nothing to prove.

\begin{lem}\label{La}
Let $(y_0, \eta_0) \in \mathcal{D}$.
For any $\varepsilon > 0$, 
there exists an open neighborhood $\Theta$ of $\eta_0$ in $M$
such that 
\begin{equation}\label{eq:La-1}
 \eta \in \Theta \ \ \mbox{implies} \ \sup_{f \in \mathbb{B}_{\Lip (X, C(K))}} 
 \bigl| (Tf)(y_0, \eta) - (Tf)(y_0, \eta_0) \bigr| < \varepsilon .
\end{equation}
\end{lem}

\begin{proof}
Put $\mathcal{D}_{y_0} = \{ \eta \in M : (y_0, \eta) \in \mathcal{D} \}$. 
Since $\eta_0 \in \mathcal{D}_{y_0}$, 
by (iv), 
there exists $j \in \{ 1, \ldots, n_{y_0} \}$ such that $\eta_0 \in 
\Omega_{y_0}^j$.
Then $\Omega_{y_0}^j$ is a clopen subset on which 
$\psi_{y_0}$ is constant.
Hence if $\eta \in \Omega_{y_0}^j$, 
then $\psi (y_0, \eta) = \psi_{y_0}(\eta) = \psi_{y_0} (\eta_0) = \psi (y_0, \eta_0)$.
Let $\varepsilon > 0$ and put 
$$
 \Theta = 
 \bigl\{ \eta \in \Omega_{y_0}^j : d_X \bigl( \varphi_{y_0} (\eta), \varphi_{y_0} (\eta_0) \bigr) < \varepsilon \bigr\} .
$$
Since $\varphi_{y_0} : \mathcal{D}_{y_0} \rightarrow X$ is continuous, 
$\Theta$ is an open neighborhood of $\eta_0$ in $\mathcal{D}_{y_0}$.
For any $\eta \in \Theta$, put $x = \varphi(y_0, \eta)$, 
$x_0 = \varphi(y_0, \eta_0)$ and $\xi = \psi (y_0, \eta) = \psi (y_0, \eta_0)$.
Then, for any $f \in \mathbb{B}_{\Lip (X, C(K))}$, 
we have
\begin{equation*}\begin{split}
 \bigl| (Tf)(y_0, \eta) - (Tf)(y_0, \eta_0) \bigr| 
 & = \bigl| 
       f \bigl( \varphi(y_0, \eta), \psi(y_0, \eta) \bigr) 
       - f \bigr( \varphi(y_0, \eta_0), \psi(y_0, \eta_0) \bigr) \bigr| \\
 & = \bigl| f(x, \xi) - f(x_0, \xi) \bigr|
   = \bigl| f_x (\xi) - f_{x_0} (\xi) \bigr| \\
 & \leq \| f_x - f_{x_0} \|_{C(K)} \\
 & \leq \mathcal{L}_{X, C(K)} (f) \: d_X (x, x_0) \\
 & = \mathcal{L}_{X, C(K)} (f) \: 
     d_X \bigl( \varphi (y_0, \eta), \varphi (y_0, \eta_0) \bigr) \\
 & = \mathcal{L}_{X, C(K)} (f) \: 
     d_X \bigl( \varphi_{y_0} (\eta), \varphi_{y_0} (\eta_0) \bigr) \\
 & \leq \| f \|_{\Lip (X, C(K))} \: \varepsilon \leq \varepsilon .
\end{split}\end{equation*}
Hence we obtain (\ref{eq:La-1}).
\end{proof}

\begin{lem}\label{Lb}
In $C(Y \times M)$, the closure of $T \bigl( \mathbb{B}_{\Lip (X, C(K))} \bigr)$ is compact.
\end{lem}

\begin{proof}
According to Arzel\'{a}-Ascoli theorem (\cite[Theorem IV.6.7]{DS}), we show that 
$T \bigl( \mathbb{B}_{\Lip (X, C(K))} \bigr)$ is bounded and equicontinuous on $Y \times M$.

The boundedness follows from an easy computation:
$$
 \bigl| (Tf)(y, \eta) \bigr| 
  \leq \| Tf \|_{C(Y \times M)} 
  \leq \| Tf \|_{\Lip (Y, C(M))}  
  \leq \| T \| \, \| f \|_{\Lip (X, C(K))} 
  \leq \| T \|  
$$
for all $(y, \eta) \in Y \times M$ and all $f \in \mathbb{B}_{\Lip (X, C(K))}$.

The equicontinuity will be shown as follows:
Clearly, $T \bigl( \mathbb{B}_{\Lip (X, C(K))} \bigr)$ is equicontinuous on the clopen set $(Y \times M) \smallsetminus \mathcal{D}$, 
because $Tf = 0$ on $(Y \times M) \smallsetminus \mathcal{D}$
for all $f \in \Lip (X, C(K))$, by (\ref{eq:Form}).
To show that $T \bigl( \mathbb{B}_{\Lip (X, C(M))} \bigr)$ is equicontinuous at each $(y_0, \eta_0) \in \mathcal{D}$, let $\varepsilon > 0$.
Take an open neighborhood $\Theta$ of $\eta_0$ in $M$ as in Lemma \ref{La}, and put $V = \bigl\{ y \in Y : d_Y (y, y_0) < \varepsilon / \| T \| \bigr\}$.
Define an open neighborhood $\mathcal{W}$ of $(y_0, \eta_0)$ in $Y \times M$ as 
$$
 \mathcal{W} = ( V \times \Theta ) \ \cap \ \mathcal{D} .
$$
Then, for any $(y, \eta) \in \mathcal{W}$ and $f \in \mathbb{B}_{\Lip (X, C(K))}$, we have 
\begin{equation*}\begin{split}
 \bigl| (Tf)(y, \eta) & - (Tf)(y_0, \eta) \bigr| 
   \leq \bigl\| (Tf)_y - (Tf)_{y_0} \bigr\|_{C(M)} 
     \leq  \mathcal{L}_{Y, C(M)} (Tf) \: d_Y (y, y_0)  \\
  & \leq  \| Tf \|_{\Lip (Y, C(M))} \: \bigl( \varepsilon / \| T \| \bigr) 
     \leq  \| T \| \, \| f \|_{\Lip (X, C(K))} \: \bigl( \varepsilon / \| T \| \bigr) \leq  \varepsilon ,
\end{split}\end{equation*}
because $y \in V$, while Lemma \ref{La} implies 
$$
\bigl| (Tf)(y_0, \eta) - (Tf)(y_0, \eta_0) \bigr|  < \varepsilon,
$$
because $\eta \in \Theta$. 
Hence the triangle inequality shows that 
$$
 (y, \eta) \in \mathcal{W} \ \ \mbox{implies} \ \sup_{f \in \mathbb{B}_{\Lip(X, C(K))}}
 \bigl| (Tf)(y, \eta) - (Tf)(y_0, \eta_0) \bigr| < 2\varepsilon.
$$
Thus we conclude that $T \bigl( \mathbb{B}_{\Lip (X, C(K))} \bigr)$ is equicontinuous on $Y\times M$. 
\end{proof}

\begin{lem}\label{Lc}
For any $\varepsilon > 0$, there exists a constant $c_{\varepsilon} > 0$ such that 
\begin{equation}\label{eq:Lc-1}
 \| Tf \|_{\Lip (Y, C(M))}  \leq  \varepsilon + c_{\varepsilon} \| Tf \|_{C(Y \times M)} 
\end{equation}
for all $f \in \mathbb{B}_{\Lip (X, C(K))}$.
\end{lem}

\begin{proof}
Fix $\varepsilon > 0$.
By (iii), there exists a $\delta_{\varepsilon} > 0$ such that 
\begin{equation}\label{eq:Lc-2}
 (y, \eta), (y', \eta) \in \mathcal{D} \ \mbox{and} \ 0 < d_Y (y, y') < \delta_{\varepsilon} \ \mbox{imply} \
 \frac{d_X \bigl( \varphi (y, \eta), \varphi (y' \eta) \bigr)}{d_Y (y, y')} < \varepsilon . 
\end{equation}
Let $f \in \mathbb{B}_{\Lip (X, C(K))}$, and choose $(y, \eta), (y', \eta) \in Y \times M$ with $y \neq y'$.
We consider three cases.

[\textsf{Case 1}] $(y, \eta), (y', \eta) \in \mathcal{D}$ : 
By (ii) in Theorem \ref{Thm1}, $y \in V^{\eta}_i$ and $y' \in V^{\eta}_{i'}$ for some $i, i' \in \{ 1, \ldots, n_{\eta} \}$.
We first consider the case $i = i'$.
If $d_Y (y, y') < \delta_{\varepsilon}$, then the computation (\ref{eq:L2-3}) 
using (\ref{eq:Lc-2}) instead of (\ref{eq:Phi}) gives
\begin{equation}\label{eq:Lc-3}\begin{split}
 \bigl| (Tf)(y, \eta) - (Tf)(y', \eta) \bigr| 
  & \leq  \mathcal{L}_{X, C(K)} (f) \, d_X \bigl( \varphi (y, \eta), \varphi (y', \eta) \bigr)  \\
  & \leq  \mathcal{L}_{X, C(K)} (f) \: \varepsilon \: d_Y (y, y')  \\
  & \leq  \| f \|_{\Lip (X, C(K))} \: \varepsilon \: d_Y (y, y')  
     \leq  \varepsilon \: d_Y (y, y') .
\end{split}\end{equation}
On the other hand, if $d_Y (y, y') \geq \delta_{\varepsilon}$, then 
\begin{equation}\label{eq:Lc-4}
 \frac{\bigl| (Tf)(y, \eta) - (Tf)(y', \eta) \bigr|}{d_Y (y, y')}  \leq  \frac{|(Tf)(y, \eta)| + |(Tf)(y', \eta)|}{\delta_{\varepsilon}} 
  \leq  \frac{2 \| Tf \|_{C(Y \times M)}}{\delta_{\varepsilon}} . 
\end{equation}
In case that $i \neq i'$, 
we have $d_Y (y, y') \geq r$ by (\ref{eq:SetDist}), %in Theorem \ref{Thm1}, 
and so 
\begin{equation}\label{eq:Lc-5}
 \frac{\bigl| (Tf)(y, \eta) - (Tf)(y', \eta) \bigr|}{d_Y (y, y')}  \leq  \frac{2 \| Tf \|_{C(Y \times M)}}{r} . 
\end{equation} 

[\textsf{Case 2}] $(y, \eta) \in \mathcal{D}$ and $(y', \eta) \in (Y \times M) \smallsetminus \mathcal{D}$ : 
Then Lemma \ref{L1} says that $d_Y (y, y') \geq \rho$ and so 
\begin{equation}\label{eq:Lc-6}
 \frac{\bigl| (Tf)(y, \eta) - (Tf)(y', \eta) \bigr|}{d_Y (y, y')}  \leq  \frac{| (Tf)(y, \eta) |}{\rho} \leq  \frac{\| Tf \|_{C(Y \times M)}}{\rho} .
\end{equation} 
 
 [\textsf{Case 3}] $(y, \eta), (y', \eta) \in (Y \times M) \smallsetminus \mathcal{D}$ : By (\ref{eq:Form}), 
\begin{equation}\label{eq:Lc-7}
 (Tf)(y, \eta) - (Tf)(y' \eta) =  0 .
\end{equation}

Now, put $\tilde{c}_{\varepsilon} = \max \bigl\{ 2 / \delta_{\varepsilon}, 2/r, 1/\rho \bigr\}$.
We combine (\ref{eq:Lc-3})--(\ref{eq:Lc-7}) to get 
\begin{equation*}
\mathcal{L}_{Y, C(M)} (Tf) 
   =   \sup_{y, y' \in Y \atop y \neq y'} \sup_{\eta \in M}  \frac{\bigl| (Tf)(y, \eta) - (Tf)(y' \eta) \bigr|}{d_Y (y, y')}  \\ 
   \leq  \max \bigl\{ \varepsilon,  \tilde{c}_{\varepsilon} \, \| Tf \|_{C(Y \times M)} \bigr\}.
\end{equation*}
Hence
$$
 \| Tf \|_{\Lip (Y, C(M))}
 % = \| Tf \|_{C(Y \times M)} + \mathcal{L}_{Y, C(M)} (Tf) \\
  \leq \varepsilon + ( \tilde{c}_{\varepsilon} + 1 ) \| Tf \|_{C(Y \times M)},
$$
which is (\ref{eq:Lc-1}).
\end{proof}

\begin{lem}\label{Ld}
In $\Lip (Y, C(M))$, the closure of $T \bigl( \mathbb{B}_{\Lip (X, C(K))} \bigr)$ is compact.
\end{lem}

\begin{proof}
Let $\{ f_n \}$ be an arbitrary sequence in $\mathbb{B}_{\Lip (X, C(K))}$.
By Lemma \ref{Lc}, there exist a subsequence $\{ f_{n_i} \}$ and 
a function $g \in C(Y \times M)$ such that 
$\bigl\| T f_{n_i} - g \bigr\|_{C(Y \times M)} \rightarrow 0$.
To see that $\bigl\{ T f_{n_i} \bigr\}$ is a Cauchy sequence in 
$\Lip (Y, C(M))$, let $\varepsilon > 0$.
Since $\bigl\{ T f_{n_i} \bigr\}$ is a Cauchy sequence in $C(Y \times M)$, 
there exists an $N$ such that $i, j \geq N$ implies 
$\bigl\| T f_{n_i} - T f_{n_j} \bigr\|_{C(Y \times M)} < \varepsilon / c_{\varepsilon}$.
Substituting $f = \bigl( f_{n_i} - f_{n_j} \bigr)/2$ in (\ref{eq:Lc-1}), 
we see 
$$
 i, j \geq N \ \mbox{implies} \ 
 \bigl\| T f_{n_i} - T f_{n_j} \bigr\|_{\Lip (Y, C(M))} \leq 
 2 \varepsilon + c_{\varepsilon} \, \bigl\| T f_{n_i} - T f_{n_j} \bigr\|_{C(Y \times M)}
 < 3 \varepsilon .
$$
Hence $\bigl\{ T f_{n_i} \bigr\}$ is a Cauchy sequence in $\Lip (Y, C(M))$, 
and so it converges to some function in $\Lip (Y, C(M))$.
Thus we conclude that the closure of $T \bigl( \mathbb{B}_{\Lip (X, C(K))} \bigr)$
 is compact in $\Lip (Y, C(M))$.
\end{proof}

Lemma \ref{Ld} says that $T$ is a compact operator 
from $\Lip (X, C(K))$ into $\Lip (Y, C(M))$, 
and the ``if'' part was proved.

\subsection{{\sc Proof of Necessity}}

In the sequels, we suppose that $T$ is compact.

\begin{lem}
$\varphi$ satisfies {\rm (iii)}.
\end{lem}

\begin{proof}
Assume, to reach a contradiction, that $\varphi$ does not satisfy (iii).
Then there exist an $\varepsilon_0 > 0$ and two sequences $\{ ( y_n, \eta_n ) \}$ and $\{ ( y'_n, \eta_n ) \}$  in $\mathcal{D}$ 
such that 
$$
 0 < d_Y ( y_n, y'_n ) < \frac{1}{n^2} \qquad \mbox{and} \qquad 
 \frac{d_X \bigl( \varphi (y_n, \eta_n), \varphi (y'_n, \eta_n ) \bigr)}{d_Y (y_n, y'_n)}  \geq  \varepsilon_0 .
$$
Put $z_n = \varphi (y_n, \eta_n)$ and $z'_n = \varphi (y'_n, \eta_n)$ for $n=1, 2, \ldots$.
In order to arrange the distance $d_X$, we here introduce a function $\chi_n$ :
$$
 \chi_n (t) = \frac{1}{2n} ( 1-e^{-nt} )  \qquad ( t \in [0, \infty) ) .
$$
Clearly, $0 \leq \chi_n \leq 1/2$ and $\chi_n$ is differentiable and $\chi_n ' (t) = e^{-nt} /2$.
Define
$$
 f_n (x) = \chi_n \bigl( d_X (x, z'_n) \bigr)  \qquad  (x \in X).
$$
For any $x, x' \in X$ with $x \neq x'$, 
the mean value theorem gives a point $s_n$ between $d_X (x, z'_n)$ and $d_X (x', z'_n)$ such that
$$
 \chi_n \bigl( d_X (x, z'_n) \bigr) - \chi_n \bigl( d_X (x', z'_n) \bigr) 
 = \chi'_n (s_n) \, \big( d_X (x, z'_n) - d_X (x', z'_n ) \bigr) ,
$$
and so
$$
 | f_n (x) - f_n (x') | =  
 | \chi'_n (s_n) | \, \big| d_X (x, z'_n) - d_x (x', z'_n ) \bigr|  
    \leq \frac{e^{-ns_n}}{2} d_X (x, x') .
$$
Hence $f_n \in \Lip (X)$ and $\| f_n \|_{\Lip (X)} = \| f_n \|_{C(X)} + \mathcal{L}_{X, \mathbb{C}} (f_n) 
 \leq \frac{1}{2n} + \frac{e^{-ns_n}}{2} \leq 1$.

Now put $\hat{f}_n (x, \xi) = f_n (x)$ for all $(x, \xi) \in X \times K$.
Then $\hat{f}_n \in \Lip (X, C(K))$ and $\| \hat{f}_n \|_{\Lip (X, C(K))} \leq 1$, that is, $\hat{f}_n \in \mathbb{B}_{\Lip (X, C(K))}$.

Next we estimate the norm $\| T \hat{f}_n \|_{\Lip (Y, C(M))}$.
We use the mean value theorem again, 
we compute as follows:
\begin{equation*}\begin{split}
 \bigl| (T \hat{f}_n)(y_n, \eta_n) - (T\hat{f}_n)(y'_n, \eta_n) \bigr| 
 & = \bigl| \hat{f}_n \bigl( \varphi (y_n, \eta_n), \psi (y_n, \eta_n) \bigr) 
          - \hat{f}_n \bigl( \varphi (y'_n, \eta_n), \psi(y'_n, \eta_n \bigr) \bigr| \\
 & = \bigl| f_n (z_n) - f_n (z'_n) \bigr| % \\
% & = \bigl| \chi_n \bigl( d_X ( \varphi (y_n, \eta_n), z'_n \bigr) - \chi_n \bi% gl( d_X ( \varphi (y'_n, \eta_n), z'_n \bigr) \bigr) \bigr|  \\
% & 
  =  \bigl| \chi_n \bigl( d_X (z_n, z'_n) \bigr) - \chi_n (0) \bigr| \\
  & = \bigl| \chi'_n ( \sigma_n )  \bigr| \, \bigl| d_X (z_n, z'_n) - 0 \bigr|  \\
  & = \frac{e^{-n \sigma_n}}{2} \: d_X \bigl( \varphi (y_n, \eta_n), \varphi (y'_n, \eta_n) \bigr)  \\
  & \geq \frac{e^{-n \sigma_n}}{2} \: \varepsilon_0 \, d_Y (y_n, y'_n) ,
 \end{split}\end{equation*}
where $0 \leq \sigma_n \leq d_X (z_n, z'_n)$.
Hence 
\begin{equation}\label{eq:Le-1}
 \bigl\| T\hat{f}_n \bigr\|_{\Lip (Y, C(M))} 
  \geq  \mathcal{L}_{Y, C(M)} ( T \hat{f}_n )  
  \geq \frac{\bigl\| ( T \hat{f}_n )_{y_n} - ( T \hat{f}_n )_{y'_n} \bigr\|_{C(M)}}{d_Y (y_n, y'_n)}
  \geq  \frac{e^{-n \sigma _n}}{2} \varepsilon_0 .\
\end{equation}
While (\ref{eq:Phi}) in Theorem \ref{Thm1} implies 
$$
 0 \leq \sigma_n \leq d_X ( z_n, z'_n ) 
 = d_X \bigl( \varphi (y_n, \eta_n), \varphi (y'_n, \eta_n) \bigr)
 \leq L \: d_Y (y_n, y'_n) \leq L \frac{1}{n^2}, 
$$
and so $n \sigma_n \rightarrow 0$.
Thus (\ref{eq:Le-1}) implies 
\begin{equation}\label{eq:Le-2}
 \liminf _{n \to \infty} \bigl\| T \hat{f}_n  \bigr\|_{\Lip (Y, C(M))}   \geq  \frac{\varepsilon_0}{2}
\end{equation}

Recall that $T$ is compact.
Since $\{ \hat{f}_n \} \subset \mathbb{B}_{\Lip (X, C(K))}$, 
there exist a subsequence $\{ \hat{f}_{n_i} \}$ and a function $g \in \Lip (Y, C(M))$ such that 
$\bigl\| T \hat{f}_{n_i} - g \bigr\|_{\Lip (Y, C(M))} \rightarrow 0$.
Since $\bigl\| T \hat{f}_{n_i} - g \bigr\|_{C(Y \times M)} 
 \leq \bigl\| T \hat{f}_{n_i} - g \bigr\|_{\Lip (Y, C(M))}$, 
we have $(T \hat{f}_{n_i}) (y, \eta) \rightarrow g(y, \eta)$ 
for each $(y, \eta) \in Y \times M$.
If $(y, \eta) \in \mathcal{D}$, then 
$$
 \bigl| (T \hat{f}_{n_i}) (y, \eta) \bigr| = \bigl| \hat{f}_{n_i} \bigl( \varphi (y, \eta), \psi (y,\eta) \bigr) \bigr| 
 = \bigl| f_{n_i} \bigl( \varphi (y, \eta) \bigr) \bigr| \leq \frac{1}{2{n_i}} \ \rightarrow \ 0, 
$$
while if $(y, \eta) \in (Y \times M) \smallsetminus \mathcal{D}$, then 
$ (T \hat{f}_{n_i})(y, \eta) = 0$.
As a result, we have $g(y, \eta) = 0$ for all $(y, \eta) \in Y \times M$, 
and so 
$$
 \bigl\| T\hat{f}_{n_i} \bigr\|_{\Lip (Y, C(M))} \ \rightarrow \ 0.
$$ 
This contradicts (\ref{eq:Le-2}).
\end{proof}

Fix $y \in Y$ and put $\mathcal{D}_y = \{ \eta \in M : (y, \eta) \in \mathcal{D} \}$.

\begin{lem}\label{Lx}
For any $\eta_0 \in \mathcal{D}_y$, there exists an open neighborhood of $\eta_0$ in $\mathcal{D}_y$ on which $\psi_y$ is constant.
\end{lem}

%This lemma can be proved by the idea dealt with in \cite{Ka}.

\begin{proof}
Since $\mathcal{D}_y$ is a compact subset of $M$,
we can treat the Banach algebra $C( \mathcal{D}_y )$ 
and a projection $P$ from $\Lip (Y, C(M))$ into $C( \mathcal{D}_y )$ : 
$$
 (Pg)(\eta) = g(y, \eta)  \qquad  
 \bigl( \eta \in \mathcal{D}_y, \ g \in \Lip (Y, C(M)) \bigr).
$$
Clearly $P$ is a bounded linear operator 
from $\Lip (Y, C(M))$ into $C( \mathcal{D}_y )$.

Now put $S = PT$.
Since $T$ is compact, 
$S$ is a compact operator from $\Lip (X, C(K))$ into $C( \mathcal{D}_y )$.
Hence Arzel\'{a}-Ascoli theorem says that 
$S \bigl( \mathbb{B}_{\Lip (X,C(K))} \bigr)$ is equicontinuous on $\mathcal{D}_y$. 
Hence there exists an open neighborhood $\Theta$ of $\eta_0$ such that 
\begin{equation}\label{eq:Lx-1}
 \eta \in \Theta \ \mbox{implies} \ 
 \sup_{f \in \mathbb{B}_{\Lip (X, C(K))}} 
 \bigl| (Sf)(\eta) - (Sf)(\eta_0) \bigr| < \frac{1}{2}.
\end{equation}

%Let us show that $\psi_y$ is constant on $\Theta$.
Conversely, assume that there exist $\eta_1 \in \Theta$ such that 
$\psi_y (\eta_1) \neq \psi_y (\eta_0)$.
By Urysohn's lemma, there exists a $u \in C(M)$ such that 
$0 \leq u \leq 1$, $u \bigl( \psi_y (\eta_1) \bigr) = 1$ and 
$\psi \bigl( \psi_y (\eta_0) \bigr) = 0$.
Put $\tilde{u} (x, \xi) = u(\xi)$ for all $(x, \xi) \in X \times K$.
Then $\tilde{u} \in \mathbb{B}_{\Lip (X, C(K))}$.
Hence (\ref{eq:Lx-1}) implies 
$$
 \bigl| (S \tilde{u})(\eta_1) - (S\tilde{u})(\eta_0) \bigr| 
 < \frac{1}{2}.
$$
But 
\begin{equation*}\begin{split}
 \bigl| (S \tilde{u})(\eta_1) - (S \tilde{u})(\eta_0) \bigr| 
 & = \bigl| (PT \tilde{u})(\eta_1) - (PT \tilde{u})(\eta_0) \bigr| \\
 & = \bigl| (T \tilde{u})(y, \eta_1) - (T \tilde{u})(y, \eta_0) \bigr| \\
 & = \bigl| \tilde{u} \bigl( \varphi (y, \eta_1), \psi (y, \eta_1) \bigr)
     - \tilde{u} \bigl( \varphi (y, \eta_0), \psi (y, \eta_0) \bigr) \bigr| \\
 & = \bigl| u( \psi (y, \eta_1) ) - u( \psi (y, \eta_0) ) \bigr| 
   = \bigl| u \bigl( \psi_y (\eta_1) \bigr) - u \bigl( \psi_y (\eta_0) \bigr) \bigr| = 1
\end{split}\end{equation*}
a contradiction.
Thus we conclude that $\psi_y$ is constant on $\Theta$.\
\end{proof}

\begin{lem}\label{Ly}
$\psi$ satisfies {\rm (iv)}.
\end{lem}

\begin{proof}
For any $\eta \in \mathcal{D}_y$, put 
$$
 \Omega^{\eta} = \bigl\{ \zeta \in \mathcal{D}_y : \psi_y (\zeta) = \psi_y (\eta) \bigr\}.
$$
Clearly, $\psi_y$ is constant on $\Omega^{\eta}$.
Also, we have 
\begin{equation}\label{eq:Lg-1}
 \Omega^{\eta} \cap \Omega^{\eta'} \neq \emptyset 
 \ \Longrightarrow \ \Omega^{\eta} = \Omega^{\eta'}.
\end{equation}

Since $\psi_y$ is continuous, 
$\Omega^{\eta}$ is a closed subset of $\mathcal{D}_y$.
Also we can easily see that Lemma \ref{Lx} implies that 
$\Omega^{\eta}$ is an open subset of $\mathcal{D}_y$.
Thus $\Omega^{\eta}$ is a clopen subset of $M$.

Note that
$$
 \mathcal{D}_y = \bigcup_{\eta \in \mathcal{D}_y} \Omega^{\eta} .
$$
Since $\mathcal{D}_y$ is compact, we can select finitely many 
$\eta_1, \ldots, \eta_n \in \mathcal{D}_y$ such that 
$$
 \mathcal{D}_y = \bigcup_{i=1}^n \Omega^{\eta_i}.
$$
By (\ref{eq:Lg-1}), we may assume that $\Omega^{\eta_1}, \ldots , \Omega^{\eta_m}$ 
are disjoint.
Putting $n_y = n$ and writing $\Omega_y^i = \Omega^{\eta_i}$ ($i = 1, \ldots , n_y$), 
we obtain (iv).
\end{proof}

\section{Applications}%第5章

we consider the following three conditions:
%\begin{itemize}
%\item[(i)]
% The square operator $x \mapsto x \circ x$ .  
%\end{itemize}

Consider the case that $K$ is a one-point set. 
Then $\Lip (X, C(K))$ is isometrically isomorphic to $\Lip (X)$.
On the other hand, if $X$ is a one-point set, 
$\Lip (X, C(K))$ is isometrically isomorphic to $C(K)$.

\begin{cor}%1
Suppose that $X$ and $Y$ are compact metric spaces with metrics $d_X$ and $d_Y$ respectively.

{\rm (I)} If $T$ is a homomorphism from $\Lip (X)$ into $\Lip (Y)$, 
then there exist a clopen subset $Y_0$ of $Y$ and 
a continuous mapping $\varphi : Y_0 \rightarrow X$ with 
$$
 \sup_{y, y \in Y_0 \atop y \neq y'} 
 \frac{d_X ( \varphi (y), \varphi (y') )}{d_Y (y, y')}  <  \infty 
$$
such that $T$ has the form :
\begin{equation}\label{eq:FormA}
 (Tf) (y) = \begin{cases}
  f \bigl( \varphi (y) \bigr)  &  \bigr( y \in Y_0 \bigr) \\
 0  &  \bigl( y \in Y \smallsetminus Y_0 \bigr)
 \end{cases} 
\end{equation}
for all $f \in \Lip(X)$.
Conversely, if $Y_0$, $\varphi$ are given as above, 
then $T$ defined by {\rm (\ref{eq:FormA})} is a homomorphism from $\Lip (X)$ into $\Lip (Y)$.
Moreover, $T$ is unital if and only if $Y_0 = Y$.

{\rm (II)} 
Suppose that $T$ is a homomorphism from $\Lip (X)$ into $\Lip (Y)$ 
with the form {\rm (\ref{eq:FormA})}. 
Then $T$ is compact if and only if
$$
 \lim_{y, y' \in Y_0 \atop d_Y (y, y') \rightarrow 0}
 \frac{d_X ( \varphi (y), \varphi (y') )}{d_Y(y, y')}  =  0.
$$
\end{cor}

Now we turn to another setting.

\begin{cor}%2
Suppose that $K$ and $M$ are compact Hausdorff spaces.

{\rm (I)} If $T$ is a homomorphism from $C(K)$ into $C(M)$, 
then there exist a clopen subset $M_0$ of $M$ and 
a continuous mapping $\psi : M_0 \rightarrow K$ 
such that $T$ has the form :
\begin{equation}\label{eq:FormB}
 (Tf) (\eta) = \begin{cases}
  f \bigl( \psi (\eta) \bigr)  &  \bigr( \eta \in M_0 \bigr) \\
 0  &  \bigl( \eta \in M \smallsetminus M_0 \bigr)
 \end{cases} 
\end{equation}
for all $f \in C(K)$.
Conversely, if $M_0$, $\psi$ are given as above, 
then $T$ defined by {\rm (\ref{eq:FormB})} is a homomorphism from $C(K)$ into $C(M)$.
Moreover, $T$ is unital if and only if $M_0 = M$.

{\rm (II)} 
Suppose that $T$ is a homomorphism from $C(K)$ into $C(M)$ 
with the form {\rm (\ref{eq:FormB})}. 
Then $T$ is compact if and only if
$M_0$ is a union of finitely many clopen subset $M_1, \ldots, M_n$ 
such that $\psi$ is constant on each $M_i$ for $i=1, \ldots, n$.
Moreover, $T$ is compact if and only if $T$ has a finite rank,
\end{cor}

\begin{cor}%3
Suppose that $X$ is a compact metric space with metric $d_X$, and that $M$ is a compact Hausdorff space.

{\rm (I)} If $T$ is a homomorphism from $\Lip (X)$ into $C(M)$, 
then there exist a clopen subset $M_0$ of $M$ and 
a continuous mapping $\varphi : M_0 \rightarrow X$ 
such that $T$ has the form :
\begin{equation}\label{eq:FormC}
 (Tf) (\eta) = \begin{cases}
  f \bigl( \varphi (\eta) \bigr)  &  \bigr( \eta \in M_0 \bigr) \\
 0  &  \bigl( \eta \in M \smallsetminus M_0 \bigr)
 \end{cases} 
\end{equation}
for all $f \in \Lip(X)$.
Conversely, if $M_0$, $\varphi$ are given as above, 
then $T$ defined by {\rm (\ref{eq:FormC})} is a homomorphism from $\Lip (X)$ into $C(M)$.
Moreover, $T$ is unital if and only if $M_0 = M$.

{\rm (II)} 
Every homomorphism from $\Lip (X)$ into $C(M)$ is compact.
\end{cor}

\newpage

\begin{cor}%4
Suppose that $Y$ is a compact metric space with metric $d_Y$, and that $K$ is a compact Hausdorff space.

{\rm (I)} If $T$ is a homomorphism from $C(K)$ into $\Lip (Y)$, 
then $Y$ is a union of finitely many disjoint clopen subset $Y_0, Y_1, \ldots, Y_n$
and there exist constant mappings $\psi_i : Y_i \rightarrow K$ ($i=1, \ldots, n$) 
such that $T$ has the form :
\begin{equation}\label{eq:FormD}
 (Tf) (y) = \begin{cases}
  f \bigl( \psi_i (y) \bigr)  &  \bigr( y \in Y_i , i=1, \ldots, n \bigr) \\
 0  &  \bigl( y \in Y_0 \bigr)
 \end{cases} 
\end{equation}
for all $f \in C(K)$.
Conversely, if $Y_0, Y_1, \ldots, Y_n$, $\psi_1, \ldots , \psi_n$ are given as above, 
then $T$ defined by {\rm (\ref{eq:FormD})} is a homomorphism from $C(K)$ into $\Lip (Y)$.
Moreover, $T$ is unital if and only if $Y_0 = \emptyset$.

{\rm (II)} 
Every homomorphism from $C(K)$ into $\Lip (Y)$ has a finite rank.
\end{cor}

%\vspace{5mm}

%\begin{ack}%謝辞
%The first and fourth authors are partially supported by Grant-in-Aid for Scientific Research, Japan Society for the Promotion of Science.
%\end{ack}

\end{document}